\documentclass[11pt,reqno]{amsart}

\usepackage{amsrefs}
\usepackage{amssymb}
\usepackage{fancyhdr}
\usepackage{bbm}
\usepackage{xcolor}
\usepackage{colonequals}
\usepackage{multirow}
\usepackage{rotating}
\usepackage{hyperref}
\usepackage{fullpage}
\hypersetup{
  colorlinks=true,
  linkcolor=blue,
  citecolor=blue
}
\usepackage{mleftright}

\newtheorem{theorem}{Theorem}
\newtheorem*{theorem*}{Theorem}
\newtheorem{lemma}[theorem]{Lemma}

\newtheorem{remark}[theorem]{Remark}

\newtheorem{conj}[theorem]{Conjecture}

\theoremstyle{definition}
\newtheorem{definition}[theorem]{Definition}

\numberwithin{equation}{section}
\numberwithin{theorem}{section}

\newcommand{\A}{\mathcal{A}}
\newcommand{\B}{\mathcal{B}}
\newcommand{\C}{\mathcal{C}}
\newcommand{\gammaG}{\gamma}
\newcommand{\coeff}[2]{\left[#1\right] #2}

\newcommand{\R}{\mathbb{R}}

\newcommand{\embolden}[1]{\textbf {#1}}
\renewcommand{\epsilon}{\varepsilon}

\title{Trees and Graphs with Non Log-concave Dominating Set Sequence via AI Tools}
\author{Alina Du, Steven Heilman, Greta Panova}
\date{\today}
\thanks{
Email: alinadu@usc.edu, stevenmheilman@gmail.com, gpanova@usc.edu\\
S.H. is supported by NSF Grant CCF AF 2448108.  G.P. is supported by NSF Grant CCF AF 2302174.  
Department of Mathematics, University of Southern California, Los Angeles, CA 90089}

\begin{document}

\begin{abstract}
We give new examples of graphs and trees with dominating set sequences that are not log-concave.  These examples were generated by PatternBoost, a transformer-based reinforcement learning software developed by Charton-Ellenberg-Wagner-Williamson.  We also show: for any positive integer $m$, there exists a tree whose dominating set sequence is not log-concave for at least $m$ indices by modifying a similar construction of Bautista-Ramos for the independent set sequence. We show that a large class of caterpillar graphs has log-concave dominating set sequences. A continuous analogue of the sequence is also log-concave for all graphs.  
\end{abstract}

\maketitle

\section{Introduction}

Let $G$ be a finite undirected graph on $n$ vertices.  A dominating set in a graph is a set of vertices which, together with their neighbors, is equal to all vertices in the graph. In other words, every vertex of the graph is either in the set or is a neighbor of a vertex from that set.  For any $0\leq j\leq n$, let $d_{j}(G)$ be the number of dominating sets in $G$ with cardinality $j$.

The following was conjectured in \cite{ali14}.
\begin{conj}\label{conj0}
For any finite graph $G$, the dominating set sequence $d_{0}(G),\ldots,d_{n}(G)$ is unimodal.  That is, there exists some $0\leq m\leq n$ such that
\[
d_{0}(G)\leq d_{1}(G)\leq\cdots\leq d_{m}(G)\geq d_{m+1}(G)\geq\cdots\geq d_{n}(G).
\]
\end{conj}
This conjecture remains widely open even for trees.  Since such a basic question remains unanswered, one might reasonably question whether Conjecture \ref{conj0} is true or not.

There are not many available techniques to prove unimodality, especially in general settings without algebraic structure. A stronger property, which implies unimodality, is log-concavity, which enjoys some better developed technology including properties like real-rootedness of the corresponding polynomial, Lorentzian properties, injective maps, etc.

 The dominating set sequence is log-concave if
\[
d_{k}(G)^{2}\geq d_{k+1}(G)d_{k-1}(G),\qquad\forall\,1\leq k\leq n-1.
\]
In the attempt to disprove Conjecture~\ref{conj0}, one may try to first disprove the stronger log-concavity property, find non log-concave families of graphs with patterns and try to amplify the structures. Alternatively, for certain families of graphs, one may try to prove log-concavity as a stronger but more technologically feasible property that implies unimodality.

In this paper we produce new families of trees and examples of more general graphs with non log-concave dominating set sequence. We produce these counterexamples using PatternBoost \cite{charton24}, which is a transformer-based reinforcement learning technology.

We also show that legged caterpillars (caterpillars with at least one leaf attached to each spine vertex) have log-concave dominating set sequence in Theorem~\ref{thm:cat}. We give a general polytope approach which shows that the continuous version of this problem is true in Theorem~\ref{thm:polytope}, showing that the failures are a discrete phenomenon related to integer point counts. We also find a parametrized family of trees that breaks log-concavity at arbitrarily many places in Theorem~\ref{main2}.

\subsection{Background and motivation}

The study of unimodality and log-concavity is especially compelling because such properties frequently reflect deep underlying structure: for example, log-concavity is closely tied to negative dependence (anti-correlation) and algebraic or geometric phenomena in a wide range of settings, including matroid theory and graph polynomials (see e.g. \cite{karim18}).

Very closely related to unimodality of dominating sets is the analogous problem for independent sets. However, already in their early work, Alavi, Erdős, Malde and Schwenk~\cite{alavi87} observed that the sequence of independent sets of a graph can have an arbitrary ascent/descent pattern. They conjectured that for trees this sequence remains unimodal, perhaps since non-unimodal graphs are very dense, unlike trees. Independent sets of trees have received significant attention over the past 10 or so years, with strong evidence for unimodality~\cite{galvin2018}, but also many breaking points of log-concavity recently discovered with PatternBoost~\cite{ramos25}.

However, for dominating sets the situation remains much less understood.

The smallest graph with a non log-concave dominating set sequence has 9 vertices \cite{beaton22}.  
\begin{figure}[ht]
    \centering
    \includegraphics[width=0.3\textwidth]{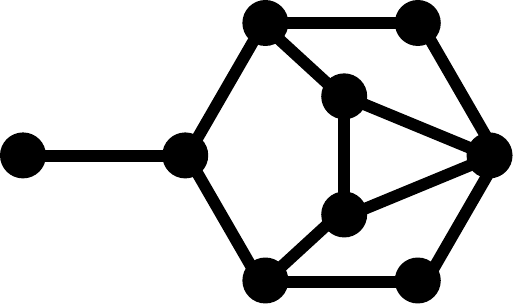}
    \caption{The unique graph on $9$ vertices with non log-concave dominating set sequence: $0,0,1,7,50,89,75,35,9,1$}
    \label{fig0}
\end{figure}

We are not aware of other published examples of such non-tree graphs.

Alikhani and Peng \cite{ali14} showed that for any $n$-vertex graph $G$,
\[
d_{0}(G)\leq\cdots\leq d_{\lceil n/2\rceil}(G).
\]

Beaton showed \cite{beaton21} that for any $n$-vertex graph without isolated vertices,
\[
d_{\lfloor 3n/4\rfloor-1}(G)\geq\cdots\geq d_{n}.
\]

It was also shown in \cite{beaton22}*{Theorem 3.2} that, if a graph $G$ on $n$ vertices has all degrees larger than $2\log_{2}n$, then $G$ has unimodal dominating set sequence with mode $\lfloor n/2\rfloor$.  That is, any counterexample to Conjecture \ref{conj0} with $n$ vertices must have a vertex with degree smaller than $2\log_{2}n$.  In this sense trees could be considered a ``difficult'' case of Conjecture \ref{conj0}, since any tree has at least two degree one vertices.  

One might ask whether random graphs exhibit log-concavity.  Dense Erd\H{o}s-R\'{e}nyi graphs have unimodal dominating set sequence with probability tending to $1$ as $n\to\infty$ (when the probability $p\in(0,1)$ of any edge appearing is fixed) \cite{beaton22}*{Theorem 3.3}.  This follows readily from Hoeffding's inequality and the $2\log_{2}n$ unimodality result, since such a random graph satisfies this degree lower bound with high probability.  For this reason, the argument of \cite{beaton22} does not extend to sparse Erd\H{o}s-R\'{e}nyi graphs, where the probability of any edge appearing goes to zero as $n\to\infty$.  In this regime, perhaps the corresponding result for independent sets could be extended to dominating sets \cite{heilman25}. 

For a while it was an open problem to find a tree whose dominating set sequence was not log-concave.  Such an example was produced in \cite{beaton24}, inspired by a tree whose independent set sequence is not log-concave \cite{kadrawi23}.  Interestingly, PatternBoost also found these same tree examples (see Section \ref{sectree}).

It was also shown in \cite{beaton24}*{Theorems 6 and 16} that, if $T$ is a tree with $n$ vertices, then
\begin{equation}\label{dtreeresults}
d_{\gamma(T)}\leq\cdots\leq d_{\lfloor(n+2\gamma(T)+1)/3\rfloor}
,\qquad 
d_{\lceil(n+2\Gamma(T)-2)/3\rceil}\geq\cdots\geq d_{n}
\end{equation}
where $\Gamma(T)$ is the size of the largest minimal dominating set in $T$, and $\gamma(T)$ is the size of the smallest dominating set in $T$.

It seems plausible that random trees would exhibit some log-concavity with high probability (by analogy with independent sets \cites{heilman25,basit21}), though we are not aware of any results of this type.

\subsection{Proofs and counterexamples}
Here we explore aspects of Conjecture~\ref{conj0}. In Section~\ref{sec:instances} we provide a geometric construction interpreting the numbers $d_j(G)$ as integer point counts in polytopes and show that their approximation via volumes gives a log-concave sequence.  We also show that a large class of caterpillar graphs we call \emph{legged caterpillars} have log-concave domination polynomials in Theorem~\ref{thm:cat}.

We produce new examples of trees and graphs with non log-concave dominating set sequence, using PatternBoost \cite{charton24}.  See Sections \ref{secgraph} and \ref{sectree} for examples of these graphs and trees.  

The application of this transformer-based reinforcement learning technology is significant since it elaborates upon known counterexamples for a difficult graph theoretic problem (Conjecture \ref{conj0}) with an automated method that does not use any explicit knowledge of the mathematical literature.  Unlike approaches to mathematics based on LLMs that use vast amounts of training data (including research papers, websites with mathematical content, etc.), PatternBoost is a reinforcement learning optimization algorithm that generates its own training data with no access to the internet or existing mathematical literature.  These examples are thus generated de novo, although some thought is needed to initialize the algorithm correctly.

Our code is available at the following GitHub link:
\begin{center}
\href{https://github.com/sheilman77/dominating_sets_PatternBoost}{https://github.com/sheilman77/dominating\underline{ }sets\underline{ }PatternBoost}
\end{center}

Lastly, without using PatternBoost, we modify a construction of \cite{bautista25} to produce a family of trees whose dominating set sequence can fail log-concavity at an arbitrarily large number of indices.

\begin{theorem}\label{main2}
For any positive integer $m\geq1$, there exists a finite tree $T=W_{m}$ such that
\[
d_{\gamma+2j+1}^{2}(T)<d_{\gamma+2j}(T)d_{\gamma+2j+2}(T),\qquad\forall\,0\leq j\leq m-1.
\]
\end{theorem}

For a picture of this tree, see Figure \ref{fig:Wmt}.  This example therefore elaborates upon the example produced in \cite{beaton24} of a tree whose dominating set sequence is not log-concave at one index.  However, Theorem \ref{main2} does not provide a path toward disproving Conjecture \ref{conj0} due to \eqref{dtreeresults}.

\section{Computational methods}

\subsection{Transformer Definition}\label{sec:transformer}
PatternBoost is a transformer-based neural network.  In this section, we therefore define a transformer.

Let $\beta>0$.  For any $x\in\R^{n}$, define the \textbf{softmax} function $p=p^{(\beta)}\colon\R^{n}\to\R^{n}$ by
\[
p_{i}(x)\colonequals\frac{e^{\beta x_{i}}}{\sum_{j=1}^{n}e^{\beta x_{j}}},\qquad\forall\,x\in\R^{n},\qquad\forall\,1\leq i\leq n.
\]

The quantity $1/\beta$ is called the \textbf{temperature} of the softmax function.

Below, we consider a function $f\colon\R^{m\times n}\to\R^{m\times n}$, i.e. $f$ is a function from $m\times n$ real matrices to $m\times n$ real matrices.  We denote the $i^{th}$ row of an $m\times n$ matrix $M$ as $M^{(i)}$, for each $1\leq i\leq m$.

\begin{definition}[\embolden{Transformer}]\label{transdef}
Let $A_{1},\ldots,A_{k},B_{1},\ldots,B_{k}$ be $n\times n$ real matrices.  A \textbf{multi-head attention layer} or \textbf{transformer} with $k$ \textbf{attention heads} is a function $f\colon \R^{m\times n}\to\R^{m\times n}$ defined by
\[
[f(Y)]^{(i)}\colonequals\sum_{j=1}^{k}p([Y A_{j} Y^{T}]^{(i)})YB_{j},\qquad\forall\, Y\in\R^{m\times n}.
\]
\end{definition}
The matrices $A_{1},\ldots,A_{k}$ are called \textbf{attention matrices}, and the matrices $B_{1},\ldots,B_{k}$ are thought of as projection matrices, though they will in general not be actual projection matrices.  Each of the $k$ terms in the sum in Definition \ref{transdef} is called an attention head.

\subsection{Background on PatternBoost}\label{secpattern}
PatternBoost works in the following way.  The input is: a reward function, a greedy local search algorithm, and a choice of representation of the object being maximized.  The output is a set of candidate maximizing objects.  In more detail:

\begin{itemize}
\item Initialize a reward function of graphs that we should maximize.  In our case, this function could be the following function of a graph $G$:
\begin{equation}\label{eq:reward}
d_{k+2}(G)d_{k}(G)-d_{k+1}(G)^{2},
\end{equation}
where $k$ is the size of the smallest dominating set in $G$.  (For a general graph $G$, the coefficients $d_{j}(G)$ cannot be computed efficiently, but they can be computed efficiently for trees.  See Section \ref{seccomp} for more details.)  

It is generally preferable that the function to maximize can be computed efficiently.
\item Initialize a greedy local search algorithm.  In our case, if we are optimizing over graphs, this means: given any graph, consider all ways to add or delete an edge, and choose one with the largest reward function (including the original graph). If we are instead optimizing over trees, we use the algorithm of \cite{ramos25}: connect two random vertices with an edge, consider all ways to subtract an edge from the resulting cycle, and among those choose a tree with the largest reward.
\item We represent graphs in our optimization as binary vectors containing the entries of the graph's adjacency matrix above the diagonal.  We represent trees in our optimization by Pr\"{u}fer codes.
\end{itemize}

After these initialization steps, the algorithm proceeds as follows.
\begin{itemize}
\item[(i)] Generate an initial list (dataset) of e.g. 50,000 graphs (or trees).  This list is made via a greedy local search, possibly with some (pseudo)random choices.
\item[(ii)] Train a transformer on the list of graphs (or trees).  More specifically, optimize the weights of the transformer using the Adam optimization method (or a variant thereof) with a loss function that fits the transformer to the initial dataset.  (We briefly define a transformer in Section \ref{sec:transformer}.)
\item[(iii)] Query the transformer to generate a new list (dataset) of e.g. 100,000 graphs (or trees).
\item[(iv)] Run the greedy local search on the new dataset.
\item[(v)] Retain the top 10\% of examples; iterate the previous three steps several times.
\end{itemize}

\subsection{Related Results}

We use the code of \cite{ramos25} who achieved similar results for the related problem for independence polynomials of trees.  Nevertheless, we made some changes to the codes of \cite{ramos25}.  Most notably, \cite{ramos25} chooses a reward function \eqref{eq:reward} (for the independent set sequence of a tree) where $k=n/2$ is a fixed function of the number of vertices of the graph.  We however found this was not necessary either for independent sets or for dominating sets.  That is, it suffices to choose the index $k$ to depend on the first nonzero index of the sequence.

Also, for trees we needed to adapt some results in the literature (albeit using standard techniques) to efficiently compute their dominating set sequences.  We describe this algorithm in Section \ref{seccomp}.

\subsection{Comparison with Exhaustive Search}

There are approximately $1.6\cdot 10^{11}$ nonisomorphic simple graphs on $12$ vertices.  Moreover, the computation of the number of dominating sets of fixed size is computationally intractable \cite{kotek17}.  So, an exhaustive search over such graphs together with computation of the dominating sets is infeasible.

The number of nonisomorphic unlabelled trees on $n$ vertices is asymptotically larger than $2.9^{n}$.  For example, there are approximately $1.4\cdot10^{10}$ such trees on $n=30$ vertices.  The number of dominating sets of a fixed size can be computed efficiently, but an exhaustive search over trees becomes intractable when $n$ is much larger than $20$.

In other words, PatternBoost does not seem to be performing a convoluted exhaustive search.

\subsection{Limitations}  Although PatternBoost is able to generate interesting examples for Conjecture \ref{conj0} (as it did in a similar problem for independent sets \cite{ramos25}), there are many problems for which PatternBoost fails to find interesting examples.  In the case of our problem, PatternBoost was unable to find a tree where log-concavity of the dominating set sequence breaks in at least two indices.  However, we know such examples exist due to Theorem \ref{main2}.  Typically an epoch for Patternboost running on an NVIDIA RTX Ada 5000 would be complete in 5 to 15 minutes.  So, running ten to twenty epochs, which was sufficient for one experimental run of PatternBoost for a graph or tree with a fixed number of vertices, could take one to several hours.  The full project, including testing runs, used about 24 to 48 hours of compute on this single GPU.  Since finding dominating sets on general graphs is inefficient, it is doubtful we could search for graphs with more than 15 to 20 vertices.  Similarly, it is doubtful we could search for trees with more than a few hundred vertices, even though their dominating set polynomials can be found efficiently.

\subsection{Conclusions from the Computations}

PatternBoost was able to generate examples of graphs and trees with non log-concave dominating set sequences de novo, replicating and extending the result of \cite{beaton24} for trees.  This result is interesting since PatternBoost generates its own data while not using the internet, any external training data, or any innate knowledge.  However, PatternBoost was unable to find a graph or tree with more than one log-concavity violation, which we were able to do ``by hand'' by adapting an argument of \cite{bautista25}.  Perhaps it is difficult in our implementation of PatternBoost to find such a highly structured tree as in Theorem \ref{main2} (and Figure \ref{fig:Wmt}).  So, our work does not rule out the possibility that PatternBoost could discover such trees, but at the present time this automated tool serves as an assistant rather than a replacement for the mathematician.

\section{Log-concave instances}\label{sec:instances}

\subsection{A Continuous Analogue}

Using the Brunn-Minkowski inequality, we demonstrate that a continuous version of the dominating set sequence of a graph is log-concave.

Let $G=(V,E)$ be an undirected graph with vertex set $V=\{1,\ldots,n\}$ and edge set $E$. Let $A$ be its adjacency matrix, i.e. $A_{i,j}=1$ if $\{i,j\}\in E$. Let $\mathbf{x}\colonequals[x_1,\ldots,x_n]^T$ and define the polytope $P(G)$ as
\[
P(G)\colonequals\{ \mathbf{x}\in[0,1]^n: (A\mathbf{x})_i+x_i \geq 1, i=1,\ldots,n\}.
\]
In other words, it consists of the points in the hypercube, such that $x_{i}+\sum_{j\in V\colon\{i,j\}\in E}x_{j}\geq1$ for every $i=1,\ldots,n$. 

\begin{lemma}
    We have that $d_k(G) = | P(G) \cap \{ \mathbf{x}\in\{0,1\}^n: x_1+\cdots+x_n=k\}|$.
\end{lemma}
\begin{proof}
    This follows immediately by treating $\mathbf{x}\in\{0,1\}^n$ as an indicator vector of a set $S \subset V$ with $x_i=1$ iff $i \in S$. The conditions of $P(G)$ are equivalent to $S$ being a dominating set: since $x_{i}+\sum_{j\in V\colon\{i,j\}\in E}x_{j}\geq1$ we have that either $i \in S$ or else there must exist at least one $j$ with $x_j=1$ and $\{i,j\} \in E$, so $j\in S$ is adjacent to $i$. The last restriction on the sum implies that $|S|=k$. 
\end{proof}

A not-so-close approximation to the number of integer points in a polytope is given by the volume. Set
\begin{equation}\label{dtildef}
\begin{aligned}
D_{k}(G)&\colonequals \{x\in P(G)\colon  \sum_{v\in V}x_{v}=k\}\\
\widetilde{d}_{k}(G)
&\colonequals\mathrm{Vol}_{n-1}(D_{k}(G))
\end{aligned}
\end{equation}
for any real number $k\in[0,n]$, where $\mathrm{Vol}_{n-1}$ denotes $(n-1)$-dimensional Lebesgue measure.  

\begin{theorem}\label{thm:polytope}
    The numbers $\widetilde{d}_k(G)$ form a log-concave sequence, i.e. 
    \[\widetilde{d}_k(G)^2 \geq \widetilde{d}_{k-1}(G) \widetilde{d}_{k+1}(G)
    \]
    for $k=1,\ldots,n-1$ and every graph $G$.
\end{theorem}

\begin{proof} 
The log-concavity of these numbers follows directly from the Brunn-Minkowski inequality.  For any real number $k$ in $[1,n-1]$, we have
\[
D_{k}(G)\supseteq \frac{1}{2}\Big(D_{k+1}(G)+ D_{k-1}(G)\Big),
\]
where $+$ denotes the Minkowski sum, and we consider $D_{k-1}(G),D_{k}(G),D_{k+1}(G)$ all as subsets of $\R^{n-1}$ by translating in the direction $(1,\ldots,1)$ as needed. 
So, using this containment, \eqref{dtildef} and the multiplicative form of Brunn-Minkowski that $\mu(\lambda A + (1-\lambda)B) \geq \mu(A)^{\lambda}\mu(B)^{1-\lambda}$ for a Lebesgue measure $\mu$, for any $0<\lambda<1$, we have (using $\lambda=1/2$)
\begin{flalign*}
\widetilde{d}_{k}(G)^{2}
&\geq \left( \mathrm{Vol}_{n-1}\Big(\frac{1}{2}(D_{k+1}(G)+  D_{k-1}(G))\Big) \right)^{2}\\
&\geq\mathrm{Vol}_{n-1}D_{k+1}(G)\cdot\mathrm{Vol}_{n-1}D_{k-1}(G)
\stackrel{\eqref{dtildef}}{=}\widetilde{d}_{k-1}(G)\widetilde{d}_{k+1}(G).
\end{flalign*}
That is, the sequence $\widetilde{d}_{0}(G),\ldots,\widetilde{d}_{n}(G)$ is log-concave.
\end{proof}

It is tempting to deduce a similar property for $d_{0}(G),\ldots,d_{n}(G)$ from this log-concavity property using e.g. the main result of \cite{harsha10}, but it would be wrong.

\subsection{Log-concave families of trees}

Paths and caterpillars are a fundamental class of trees. They were seriously studied in the context of the independence set problem, see e.g.~\cites{BES2018, galvin2018, WangZhu2011,Zhu2007}. 

A caterpillar tree $T(a_1,\ldots,a_k)$ consists of a spine of vertices $v_1,\ldots,v_k$, with edges $v_i - v_{i+1}$, and to vertex $v_i$ there are $a_i$ many leaves attached. 

\begin{theorem}\label{thm:cat}
    Let $T\colonequals T(a_1,\ldots,a_k)$ be a caterpillar tree with $a_i \geq 1$ for all $i$. Then the domination  polynomial is equal to 
    \[
    D(T;x) \colonequals \sum_{j\geq0} d_j(T)x^j=  \prod_{i=1}^k (x^{a_i}+x(1+x)^{a_i})
    \]
    and its coefficients form a log-concave sequence, i.e. $d_j(T)^2 \geq d_{j-1}(T)d_{j+1}(T)$, $\forall$ $j\geq1$.
\end{theorem}
\begin{proof}
    Let $S$ be a dominating set of $T$. If $v_i \in S$, then each of the leaves attached to $v_i$ may or may not be in $S$ independently of the rest of the graph. This contributes $x(1+x)^{a_i}$ to the domination polynomial. If $v_i \not \in S$, then each of the leaves attached to $v_i$ has to be in $S$, otherwise it would not be ``covered'' by a vertex from $S$. This gives a term $x^{a_i}$ in the domination polynomial. In that case the vertex $v_i$ is ``covered'' by its leaves (here we use $a_i\geq 1$) and the local configuration does not depend on which vertices from the rest of the graph are in $S$. Thus the choice of which vertices among $v_i$ and its attached leaves lie in $S$ is independent for each $i$, so 
    \[
    D(T;x) = \prod_{i=1}^{k} \left( x(1+x)^{a_i} + x^{a_i} \right).
    \]
    Notice that 
    \[
    x(1+x)^{a_i}+x^{a_i} = x(1 + a_ix +\binom{a_{i}}{2}x^{2}+\cdots+ (a_i+1)x^{a_i-1} + x^{a_i})
    \]
    is a log-concave polynomial since most of its coefficients are binomials which satisfy
    \[
    \binom{a_i}{j}^2 \geq \binom{a_i}{j-1}\binom{a_i}{j+1}
    \]
    thereby verifying log-concavity for the coefficients where $j \neq a_i-2$ and 
    \[\binom{a_i}{2}^2 \geq \binom{a_i}{3}(a_i+1).
    \]
    Finally, the product of log-concave polynomials is log-concave and the claim follows. 
\end{proof}

\begin{remark}
    We believe that also legless (when some $a_i=0$) caterpillars have log-concave domination polynomials. This result shows a different story to the independence polynomial log-concavity investigations. Only regular (all $a_i$ being equal) or periodic regular ($a_{2i}=a,a_{2i+1}=b$) caterpillars are known to have unimodal independence polynomials~\cites{galvin2018,BES2018}, whereas the general case remains open. 
\end{remark}

\section{Trees with Arbitrary Breaks}

In this section we prove that an arbitrary number of log-concavity failures can occur for a family built from the Bautista--Ramos trees \cite{bautista25} by adding a leaf to each existing leaf of the trees, as in \cite{beaton24}.

We work with the following rooted gadgets (see Figure \ref{fig:Wmt}.)
\begin{itemize}
    \item Let $L$ be the rooted path on three vertices $(u,v,w)$, rooted at $u$.
    \item Let $F_t$ be the rooted tree obtained by taking a root $c$ with $t$ children, each a copy of $L$.
    \item Let $H_t$ be the rooted tree obtained by taking a root $v$ with three children, each a copy of $F_t$.
    \item Let $X$ be the rooted path on two vertices $(a,b)$, rooted at $a$.
    \item Let $W_{m,t}$ be the rooted tree obtained by taking a new root $r$, attaching $m$ children each a copy of $H_t$, and attaching one additional child which is a copy of $X$.
\end{itemize}

This is exactly the tree obtained from the Bautista--Ramos tree $TG_{m,t}$ \cite{bautista25} by attaching one new leaf to every leaf.

\begin{figure}[ht]
\centering
\includegraphics[width=\textwidth]{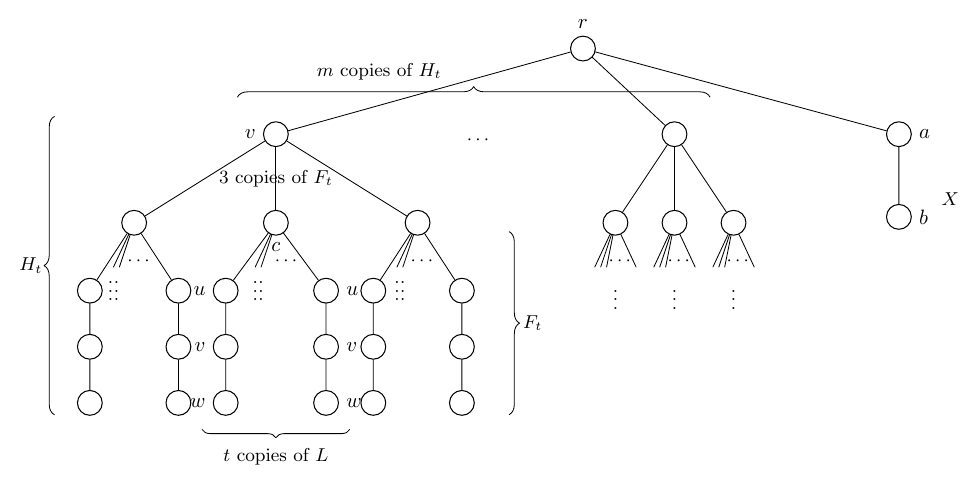}
\caption{The tree $W_{m,t}$. The top root $r$ is adjacent to $m$ copies of $H_t$ and to one additional copy of $X$. In the leftmost branch, one copy of $H_t$ is expanded; inside it, the middle copy of $F_t$ is expanded; and inside that copy of $F_t$, the rooted path $L$ is repeated $t$ times.}
\label{fig:Wmt}
\end{figure}

\subsection{Rooted domination states}\label{seccomp}
Let $S$ be a rooted tree with root $\rho$ which we view as a subtree in a larger structure.  Let $x$ be an indeterminate.  The dominating set polynomial can be split recursively into the following three parts.
\begin{itemize}
    \item $\A_S(x)$: generating polynomial for dominating sets of $S$ in which $\rho$ is chosen.
    \item $\B_S(x)$: generating polynomial for dominating sets of $S$ in which $\rho$ is not chosen and is not dominated by its parent (so it is dominated by one of its children).
    \item $\C_S(x)$: generating polynomial for dominating sets of $S$ in which $\rho$ is not chosen and is not dominated by any child (so it is dominated by its parent).
\end{itemize}
In every case, all other vertices of $S$ must already be dominated inside $S$.

If the root of $S$ has child-subtrees $S_1,\dots,S_d$, then the standard recurrences are
\begin{align}
\A_S(x) &= x\prod_{i=1}^d \bigl(\A_{S_i}(x)+\B_{S_i}(x)+\C_{S_i}(x)\bigr), \label{eq:Arec}\\
\B_S(x) &= \prod_{i=1}^d \bigl(\A_{S_i}(x)+\B_{S_i}(x)\bigr)-\prod_{i=1}^d \B_{S_i}(x), \label{eq:Brec}\\
\C_S(x) &= \prod_{i=1}^d \B_{S_i}(x). \label{eq:Crec}
\end{align}
Indeed, if the root is chosen then it dominates every child-root, so each child-root may be in any of the three states; if the root is in state $\C$, then every child-root must be in state $\B$; and if the root is in state $\B$, then every child-root must lie in state $\A$ or $\B$, with at least one child in state $\A$.

\begin{remark}
The recursions \eqref{eq:Arec}, \eqref{eq:Brec} and \eqref{eq:Crec} imply a dynamic programming algorithm to efficiently compute the dominating set sequence of a tree.  Given a tree, choose a root, then recursively compute $\A,\B$ and $\C$ for each rooted subtree, using depth-first search.
\end{remark}
\subsection{The basic branch and the t-fan gadget}

\begin{lemma}\label{lem:L}
Let $L$ be a path graph with vertices $(u,v,w)$ rooted at $u$, considered as a subgraph of a larger tree, i.e $u$ is connected to a vertex not equal to $v$ in the larger tree.  Then
\[
\A_L(x)=x^2(2+x),\qquad \B_L(x)=x(1+x),\qquad \C_L(x)=x.
\]
\end{lemma}

\begin{proof}
Direct enumeration gives the answer:
In state $\A$, the root $u$ is chosen. To dominate $w$, we must also choose either $v$ or $w$. Hence
    $\A_L = 2x^2+x^3 = x^2(2+x)$. In state $\B$, the root $u$ is not chosen but must be dominated from below, so $v$ must be chosen; then $w$ may or may not be chosen. Hence
    $\B_L = x+x^2 = x(1+x)$. In state $\C$, the root $u$ is left to be dominated by the parent, so $v$ cannot be chosen. Then $w$ must be chosen in order to dominate $v$. Hence $    \C_L = x.$
\end{proof}

Now let $F_t$ be the rooted tree whose root $c$ has $t$ children, each a copy of $L$. Define
\begin{equation}\label{eq:PQR}
P_t\colonequals(1+3x+x^2)^t,
\qquad
Q_t\colonequals(2+3x+x^2)^t,
\qquad
R_t\colonequals\frac{P_t-(1+x)^t}{x}.
\end{equation}

\begin{lemma}\label{lem:Ft}
For the rooted tree $F_t$,
\[
\A_{F_t}(x)=x^{t+1}Q_t,
\qquad
\B_{F_t}(x)=x^{t+1}R_t,
\qquad
\C_{F_t}(x)=x^t(1+x)^t.
\]
\end{lemma}

\begin{proof}
By Lemma~\ref{lem:L},
\[
\A_L+\B_L+\C_L
= x^2(2+x)+x(1+x)+x
= x(2+3x+x^2),
\]
so by \eqref{eq:Arec},
\[
\A_{F_t} 
= x\bigl(x(2+3x+x^2)\bigr)^t 
\stackrel{\eqref{eq:PQR}}{=} x^{t+1}Q_t.
\]
Also,
\[
\A_L+\B_L 
= x^2(2+x)+x(1+x) 
= x(1+3x+x^2),
\]
so by \eqref{eq:Brec} and \eqref{eq:Crec},
\[
\B_{F_t} 
= \bigl(x(1+3x+x^2)\bigr)^t - \bigl(x(1+x)\bigr)^t
= x^t\bigl((1+3x+x^2)^t-(1+x)^t\bigr)
\stackrel{\eqref{eq:PQR}}{=} x^{t+1}R_t,
\]
while by \eqref{eq:Crec}
\[
\C_{F_t}=\bigl(x(1+x)\bigr)^t=x^t(1+x)^t.
\]
\end{proof}

\subsection{The three-branch gadget}
Let $H_t$ be the rooted tree whose root has three children, each a copy of $F_t$.

\begin{lemma}\label{lem:Ht}
For $H_t$,
\[
\A_{H_t}(x)=x^{3t+1}U_t,
\qquad
\B_{H_t}(x)=x^{3t+3}V_t,
\qquad
\C_{H_t}(x)=x^{3t+3}R_t^3,
\]
where
\begin{equation}\label{uvdef}
U_t(x)\colonequals(P_t+xQ_t)^3,
\qquad
V_t(x)\colonequals(Q_t+R_t)^3-R_t^3.
\end{equation}
\end{lemma}

\begin{proof}
From Lemma~\ref{lem:Ft},
\begin{equation}\label{abcft}
\A_{F_t}+\B_{F_t}+\C_{F_t}
= x^{t+1}Q_t + x^{t+1}R_t + x^t(1+x)^t
= x^t\bigl(xQ_t + xR_t + (1+x)^t\bigr).
\end{equation}
By the definition of $R_t$ in \eqref{eq:PQR}, we have
\[
xR_t+(1+x)^t
=(1+3x+x^2)^t
\stackrel{\eqref{eq:PQR}}{=}P_t,
\]
so
\[
\A_{F_t}+\B_{F_t}+\C_{F_t}
\stackrel{\eqref{abcft}}{=}x^t(P_t+xQ_t).
\]
Hence \eqref{eq:Arec} yields
\[
\A_{H_t}=x\bigl(x^t(P_t+xQ_t)\bigr)^3=x^{3t+1}(P_t+xQ_t)^3
\stackrel{\eqref{uvdef}}{=}
x^{3t+1}U_t.
\]
Likewise,
\[
\A_{F_t}+\B_{F_t}
=x^{t+1}(Q_t+R_t),
\]
so \eqref{eq:Brec} gives
\[
\B_{H_t}=\bigl(x^{t+1}(Q_t+R_t)\bigr)^3 - \bigl(x^{t+1}R_t\bigr)^3
= x^{3t+3}\bigl((Q_t+R_t)^3-R_t^3\bigr)
\stackrel{\eqref{uvdef}}{=} x^{3t+3}V_t.
\]
Finally, by \eqref{eq:Crec}
\[
\C_{H_t}=\bigl(x^{t+1}R_t\bigr)^3 = x^{3t+3}R_t^3.
\]
\end{proof}

\subsection{Coefficient growth estimates}
We now isolate the exponential growth in $t$ of the low-degree coefficients of $U_t$ and $V_t$.

\begin{lemma}\label{lem:coefbasic}
Fix $r\ge 0$.
\begin{enumerate}
    \item $\coeff{x^r}{P_t}$ and $\coeff{x^r}{R_t}$ are polynomials in $t$.
    \item There is a polynomial $q_r(t)$ with positive leading coefficient such that
    \[
    \coeff{x^r}{Q_t}=2^t q_r(t).
    \]
    In particular, $\coeff{x^r}{Q_t}=2^t t^{O(1)}$.
\end{enumerate}
As a consequence,
\[
\coeff{x^r}{V_t}=2^{3t} t^{O(1)},
\qquad
\coeff{x^r}{U_t}=2^{\min(r,3)t} t^{O(1)}.
\]
\end{lemma}

\begin{proof}
For $P_t=(1+3x+x^2)^t$, the coefficient of $x^r$ is obtained by choosing in each factor one of $1$, $3x$, or $x^2$, with total degree $r$. For fixed $r$, the number of contributing choices is finite and independent of $t$, and each contribution is a product of binomial coefficients in $t$; hence $\coeff{x^r}{P_t}$ is a polynomial in $t$.

Since
\[
R_t \stackrel{\eqref{eq:PQR}}{=} \frac{P_t-(1+x)^t}{x},
\]
we have
\[
\coeff{x^r}{R_t} = \coeff{x^{r+1}}{P_t} - \coeff{x^{r+1}}{(1+x)^t},
\]
and each term on the right is polynomial in $t$, proving part (1).

For part (2), write
\[
Q_t=(2+3x+x^2)^t 
\stackrel{\eqref{eq:PQR}}{=}
2^t\left(1+\frac32 x + \frac12 x^2\right)^t.
\]
The coefficient of $x^r$ in the second factor is again a polynomial in $t$ for fixed $r$, and its leading term is positive because the term $(\tfrac32 x)^r$ contributes positively. Thus
\[
\coeff{x^r}{Q_t}=2^t q_r(t)
\]
for a polynomial $q_r(t)$ with positive leading coefficient.

Now,
\[
V_t
\stackrel{\eqref{uvdef}}{=}
(Q_t+R_t)^3-R_t^3 
= Q_t^3 + 3R_tQ_t^2 + 3R_t^2Q_t.
\]
By part (2), every fixed low-degree coefficient of $Q_t$ is $2^t t^{O(1)}$, while by part (1) every fixed low-degree coefficient of $R_t$ is $t^{O(1)}$. Hence the coefficient of $x^r$ in $Q_t^3$ has size $2^{3t}t^{O(1)}$, and the coefficients in the remaining two summands are at most $2^{2t}t^{O(1)}$ and $2^t t^{O(1)}$. Therefore
\[
\coeff{x^r}{V_t}=2^{3t}t^{O(1)}.
\]

For $U_t\stackrel{\eqref{uvdef}}{=}(P_t+xQ_t)^3$, expand
\[
U_t = P_t^3 + 3P_t^2(xQ_t) + 3P_t(xQ_t)^2 + (xQ_t)^3.
\]
Each factor $xQ_t$ contributes one unit of degree and one factor $2^t$ to the exponential order. Therefore, in degree $r$, at most $\min(r,3)$ such factors can appear, and the largest exponential order is $2^{\min(r,3)t}$. This proves
\[
\coeff{x^r}{U_t}=2^{\min(r,3)t}t^{O(1)}.
\]
\end{proof}

The next lemma controls low-degree coefficients of powers of $U_t$.

\begin{lemma}\label{lem:Upowers}
Fix integers $q\ge 0$ and $0\le u\le 2q$. Then
\[
\coeff{x^u}{U_t^q}=2^{ut}t^{O(1)}.
\]
\end{lemma}

\begin{proof}
For the upper bound, use convolution together with Lemma~\ref{lem:coefbasic}. Writing
\begin{equation}\label{usum}
\coeff{x^u}{U_t^q} = \sum_{
\substack{
u_{1},\ldots,u_{q}\geq0\colon\\
u_1+\cdots+u_q=u}} \prod_{i=1}^q \coeff{x^{u_i}}{U_t},
\end{equation}
each factor satisfies
\[
\coeff{x^{u_i}}{U_t}=2^{\min(u_i,3)t}t^{O(1)}.
\]
Since $\sum_{i=1}^{q}u_{i}=u\le 2q$, every $u_i$ appearing in a nonnegative composition of $u$ satisfies $\sum_{i=1}^{q} \min(u_i,3)\le u$, and therefore each summand in \eqref{usum} is at most $2^{ut}t^{O(1)}$. Summing over finitely many compositions gives
\[
\coeff{x^u}{U_t^q}\le 2^{ut}t^{O(1)}.
\]

For the lower bound, write
\[
u = \varepsilon_1+\cdots+\varepsilon_q
\]
with each $\varepsilon_i\in\{0,1,2\}$; this is possible because $u\le 2q$. The corresponding product term in the convolution contributes
\[
\prod_{i=1}^q \coeff{x^{\varepsilon_i}}{U_t}.
\]
By Lemma~\ref{lem:coefbasic}, for $\varepsilon_i\in\{0,1,2\}$ we have
\[
\coeff{x^{\varepsilon_i}}{U_t}=2^{\varepsilon_i t}t^{O(1)},
\]
with positive leading constant. Hence this product is
\[
2^{(\varepsilon_1+\cdots+\varepsilon_q)t}t^{O(1)}=2^{ut}t^{O(1)}.
\]
Since all coefficients are nonnegative, this yields the matching lower bound. Therefore
\[
\coeff{x^u}{U_t^q}=2^{ut}t^{O(1)}.
\]
\end{proof}

\subsection{The whole tree}
We now incorporate the extra whiskered leaf at the top.

\begin{lemma}\label{lem:X}
Let $X$ be a path graph with vertices $(a,b)$ rooted at $a$.  Then
\[
\A_X(x)=x(1+x),\qquad \B_X(x)=x,\qquad \C_X(x)=0.
\]
\end{lemma}

\begin{proof}
In state $\A$, either $\{a\}$ or $\{a,b\}$ is chosen, so $\A_X=x+x^2=x(1+x)$. In state $\B$, the root $a$ is not chosen and must be dominated by its child, so $b$ must be chosen; thus $\B_X=x$. State $\C$ is impossible, because if $a$ is left to be dominated by its parent then there is no way to dominate $b$ without choosing $a$ or $b$, and choosing $b$ would force $a$ into state $\B$ instead. Hence $\C_X=0$.
\end{proof}

\begin{lemma}\label{lem:gamma}
For the tree $W_{m,t}$, its smallest dominating set has cardinality $\gammaG(W_{m,t})$ where
\[
\gammaG(W_{m,t}) = m(3t+1)+1.
\]
\end{lemma}

\begin{proof}
Each branch $H_t$ contributes at least $3t+1$ vertices in any dominating set meeting the branch, because $\A_{H_t}$ starts in degree $3t+1$ and both $\B_{H_t}$ and $\C_{H_t}$ start in degree $3t+3$ by Lemma~\ref{lem:Ht}. The extra branch $X$ contributes at least one vertex by Lemma~\ref{lem:X}. Hence every dominating set of $W_{m,t}$ has size at least $m(3t+1)+1$.

This lower bound is attained by leaving the top root $r$ unchosen, putting the $X$-branch in state $\B$, and putting each $H_t$-branch in state $\A$. Therefore
\[
\gammaG(W_{m,t}) = m(3t+1)+1.
\]
\end{proof}

Write
\[
D(W_{m,t};x)=x^{\gammaG(W_{m,t})}\sum_{r\ge 0} e_r(m,t)x^r.
\]
Thus $e_r(m,t)=d_{\gammaG(W_{m,t})+r}(W_{m,t})$.

\begin{lemma}\label{lem:er-growth}
Fix $m\ge 1$. For every integer $r$ with $0\le r\le 2m$,
\[
e_r(m,t)=2^{v_r t} t^{O(1)},
\qquad\text{where}\qquad
v_r = r+\left\lfloor \frac r2 \right\rfloor,
\]
where the implicit constants can depend on $m,r$ but not $t$.
\end{lemma}

\begin{proof}
We analyze the contributions from the top root in states $\B$ and $\A$ separately.

\smallskip
\noindent\textbf{Root in state $\B$.}
Using \eqref{eq:Brec} with one child-subtree $X$ and $m$ child-subtrees $H_t$, and recalling $\C_X=0$, we obtain
\[
\B_{W_{m,t}} = (\A_X+\B_X)(\A_{H_t}+\B_{H_t})^m - \B_X\,\B_{H_t}^m.
\]
By Lemmas~\ref{lem:Ht} and \ref{lem:X},
\[
\A_X+\B_X = 2x+x^2 = x(2+x),
\]
\[
\A_{H_t}+\B_{H_t}=x^{3t+1}U_t + x^{3t+3}V_t = x^{3t+1}(U_t+x^2V_t),
\]
\[
\B_{H_t}=x^{3t+3}V_t.
\]
Since $\gammaG(W_{m,t})=m(3t+1)+1$ by Lemma~\ref{lem:gamma}, this gives
\begin{equation}\label{eq:BW}
\B_{W_{m,t}} = x^{\gammaG(W_{m,t})}\Bigl((2+x)(U_t+x^2V_t)^m - x^{2m}V_t^m\Bigr).
\end{equation}
Expand
\begin{equation}\label{eq:expand-main}
(U_t+x^2V_t)^m = \sum_{j=0}^m \binom{m}{j} x^{2j} U_t^{m-j}V_t^j.
\end{equation}
Fix $r$ with $0\le r\le 2m$. In the term indexed by $j$, the coefficient of $x^r$ is governed by
\[
\coeff{x^{r-2j}}{U_t^{m-j}V_t^j}.
\]
By Lemmas~\ref{lem:coefbasic} and \ref{lem:Upowers}, this coefficient has size
\[
2^{(r-2j)t}t^{O(1)}\cdot 2^{3jt}t^{O(1)} = 2^{(r+j)t}t^{O(1)},
\]
provided $r-2j\ge 0$; otherwise the contribution is zero. Thus among all $j$ with $2j\leq r$, the exponential order is maximized when $j=\lfloor r/2\rfloor$, and the maximal exponent is
\[
r+\left\lfloor \frac r2\right\rfloor = v_r.
\]
Therefore the coefficient of $x^r$ in the first term of \eqref{eq:BW} has size $2^{v_r t}t^{O(1)}$.

The subtraction term $x^{2m}V_t^m$ in \eqref{eq:BW} only affects degree $r=2m$ among the range $0\le r\le 2m$. In that degree, the factor $(2+x)$ in the first term contributes two copies of the dominant $x^{2m}V_t^m$ contribution, so subtracting one copy leaves a positive contribution of the same exponential order. Hence the coefficient of $x^r$ in the normalized root-$\B$ contribution has size
\[
2^{v_r t}t^{O(1)}
\qquad (0\le r\le 2m).
\]

\smallskip
\noindent\textbf{Root in state $\A$.}
From \eqref{eq:Arec},
\[
\A_{W_{m,t}} = x(\A_X+\B_X+\C_X)(\A_{H_t}+\B_{H_t}+\C_{H_t})^m.
\]
By Lemmas~\ref{lem:Ht} and \ref{lem:X},
\[
\A_X+\B_X+\C_X = x(2+x),
\]
and
\[
\A_{H_t}+\B_{H_t}+\C_{H_t}
= x^{3t+1}U_t + x^{3t+3}(V_t+R_t^3)
= x^{3t+1}\bigl(U_t + x^2(V_t+R_t^3)\bigr).
\]
Thus
\begin{equation}\label{eq:AW}
\A_{W_{m,t}} = x^{\gammaG(W_{m,t})+1}(2+x)\bigl(U_t+x^2(V_t+R_t^3)\bigr)^m.
\end{equation}
Compared with \eqref{eq:BW}, there is an extra factor of $x$, so the contribution to $e_r(m,t)$ coming from \eqref{eq:AW} has exponential order at most $2^{v_{r-1}t}t^{O(1)}$ for $r\ge 1$, which is strictly smaller than $2^{v_rt}t^{O(1)}$ because $v_r>v_{r-1}$. For $r=0$, the root-$\A$ contribution vanishes anyway because of the extra factor of $x$.

Combining the two cases shows that the total normalized coefficient satisfies
\[
e_r(m,t)=2^{v_r t}t^{O(1)}
\qquad (0\le r\le 2m),
\]
as claimed.
\end{proof}

\subsection{Main Theorem: Arbitrarily Many Log-Concavity Violations}

We now restate Theorem \ref{main2} as Theorem \ref{thm:main} below, and prove it.

\begin{theorem}\label{thm:main}
Fix $m\ge 1$. Then for all sufficiently large $t$, the domination polynomial of $W_{m,t}$ has at least $m$ log-concavity failures. More precisely, if
\[
D(W_{m,t};x)=\sum_{s\ge \gammaG(W_{m,t})} d_s x^s,
\]
then for every $j=0,1,\dots,m-1$ one has
\[
d_{\gammaG(W_{m,t})+2j+1}^2 < d_{\gammaG(W_{m,t})+2j}\,d_{\gammaG(W_{m,t})+2j+2}
\]
for all sufficiently large $t$.
\end{theorem}
\begin{proof}
By Lemma~\ref{lem:er-growth}, for each $0\le j\le m-1$,
\[
e_{2j}(m,t)=2^{(2j+j)t}t^{O(1)}=2^{3jt}t^{O(1)},
\]
\[
e_{2j+1}(m,t)=2^{(2j+1+j)t}t^{O(1)}=2^{(3j+1)t}t^{O(1)},
\]
\[
e_{2j+2}(m,t)=2^{(2j+2+j+1)t}t^{O(1)}=2^{(3j+3)t}t^{O(1)}.
\]
Therefore
\[
e_{2j}(m,t)e_{2j+2}(m,t)=2^{(6j+3)t}t^{O(1)},
\]
while
\[
e_{2j+1}(m,t)^2=2^{(6j+2)t}t^{O(1)}.
\]
Consequently, as $t\to\infty$,
\[
\frac{e_{2j}(m,t)e_{2j+2}(m,t)}{e_{2j+1}(m,t)^2}=2^t t^{O(1)}\to\infty.
\]
Hence for all sufficiently large $t$,
\[
e_{2j+1}(m,t)^2 < e_{2j}(m,t)e_{2j+2}(m,t).
\]
Since $e_r(m,t)=d_{\gammaG(W_{m,t})+r}(W_{m,t})$, this is exactly
\[
d_{\gammaG(W_{m,t})+2j+1}^2 < d_{\gammaG(W_{m,t})+2j}\,d_{\gammaG(W_{m,t})+2j+2},
\]
for each $j=0,1,\dots,m-1$. Thus the domination polynomial has at least $m$ log-concavity failures, occurring at the indices
\[
\gammaG(W_{m,t})+1,\ \gammaG(W_{m,t})+3,\ \dots,\ \gammaG(W_{m,t})+2m-1.
\]
\end{proof}

\begin{remark}
The theorem proves the existence of at least $m$ failure indices. It does not assert that these are the only failures.
\end{remark}

\newpage 

\section{Examples of Graphs}\label{secgraph}

\begin{figure}[ht!]
    \centering
    \includegraphics[width=\textwidth]{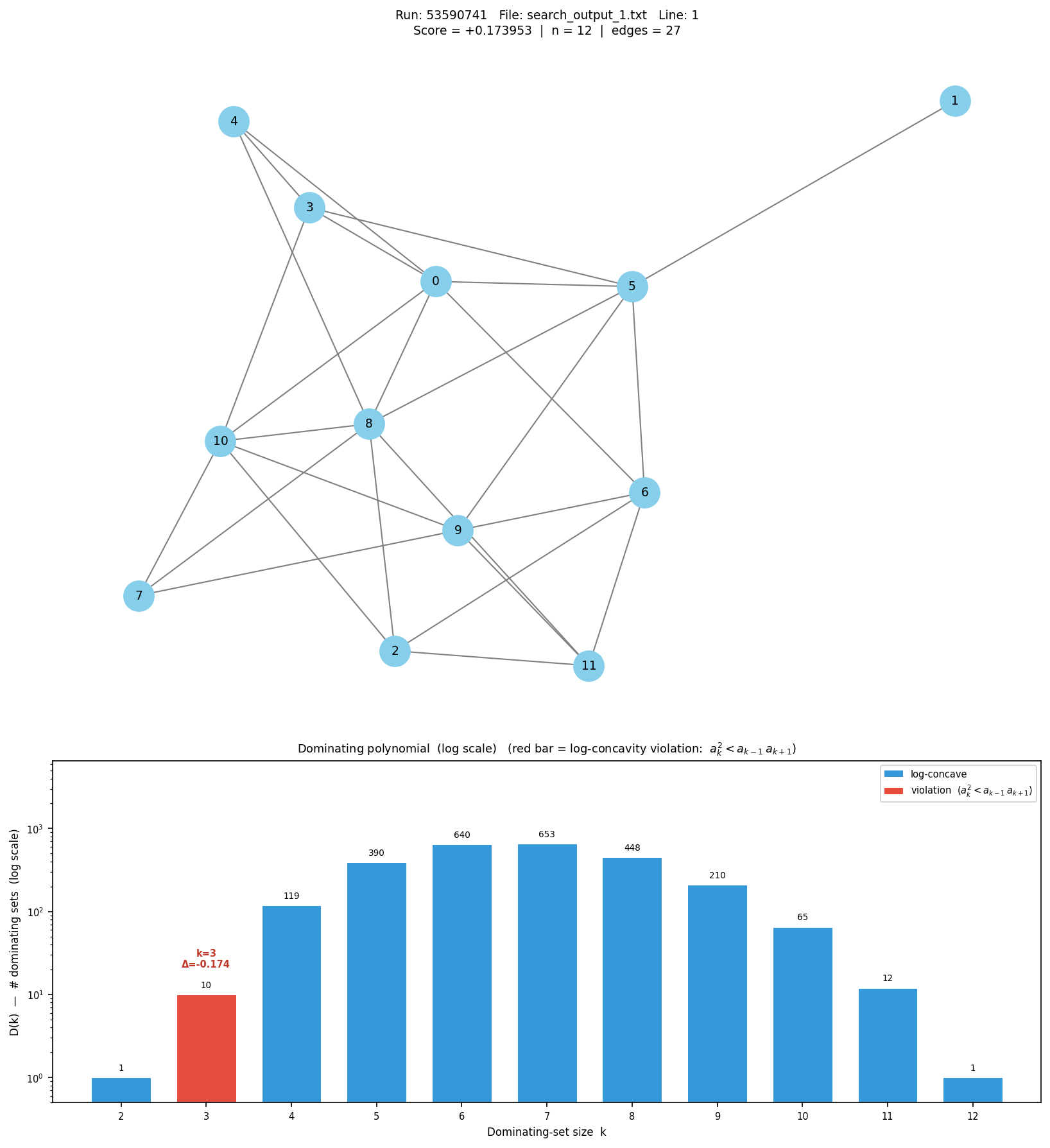}
    \caption{A 12-vertex graph with a non log-concave dominating set sequence}
    \label{fig1}
\end{figure}

\begin{figure}[ht!]
    \centering
    \includegraphics[width=\textwidth]{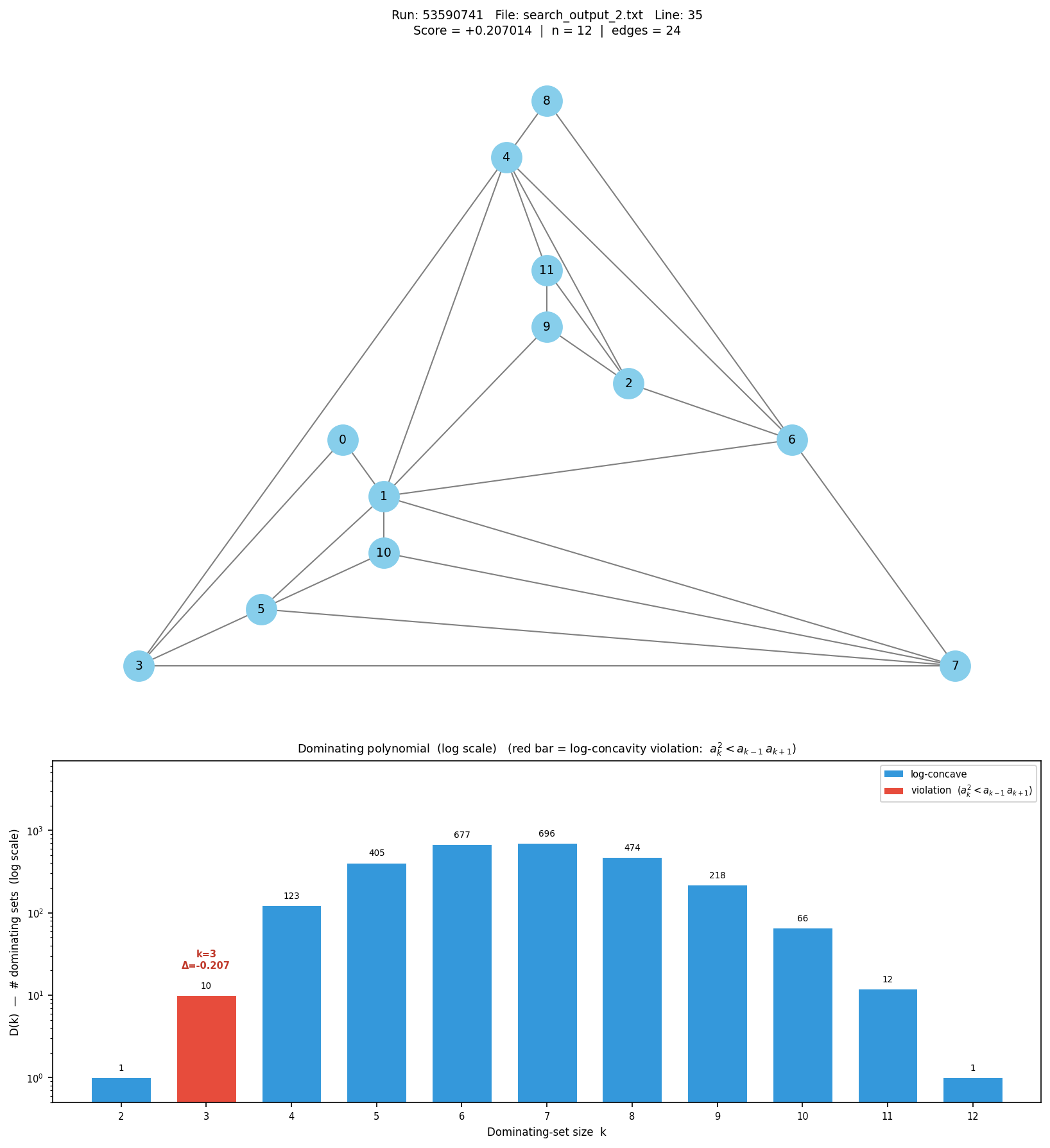}
\caption{A 12-vertex graph with a non log-concave dominating set sequence}
    \label{fig2}
\end{figure}

\begin{figure}[ht!]
    \centering
    \includegraphics[width=\textwidth]{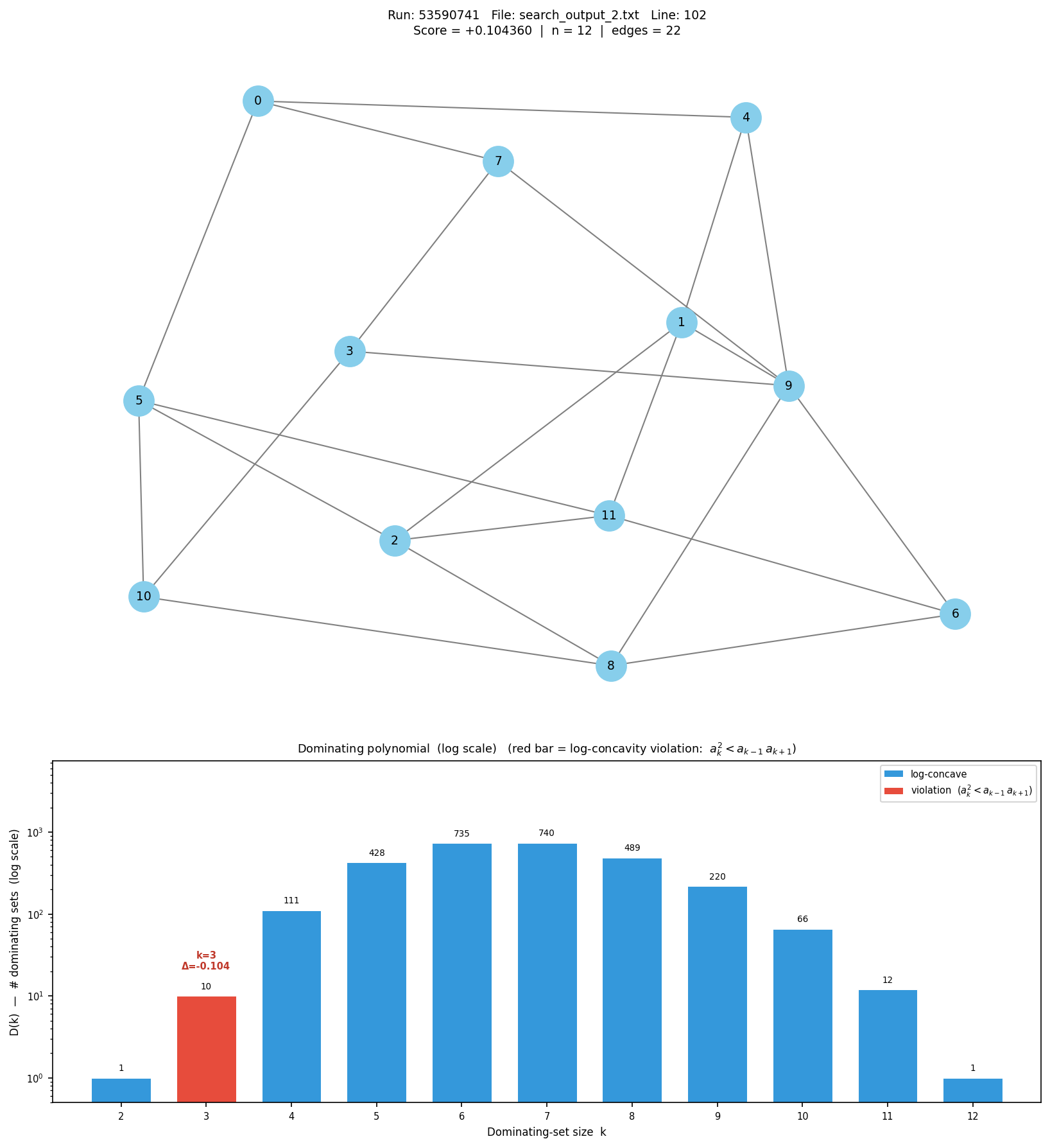}
  \caption{A 12-vertex graph with a non log-concave dominating set sequence}
    \label{fig3}
\end{figure}

\clearpage

\section{Examples of Trees}\label{sectree}

\begin{figure}[ht!]
    \centering
    \includegraphics[width=\textwidth]{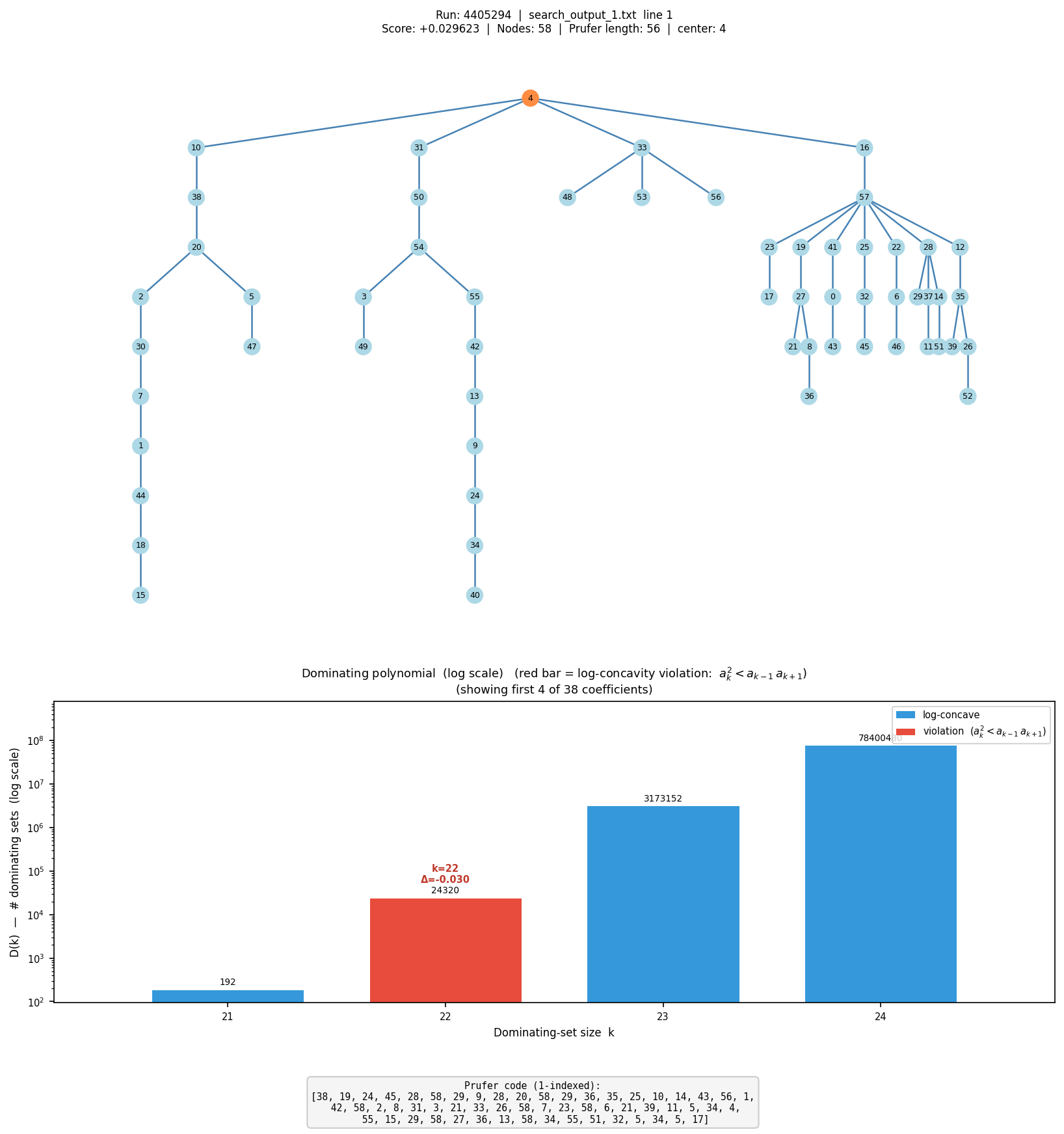}
  \caption{A 58-vertex tree with a non log-concave dominating set sequence}
    \label{fig4}
\end{figure}

\begin{figure}[ht!]
    \centering
    \includegraphics[width=\textwidth]{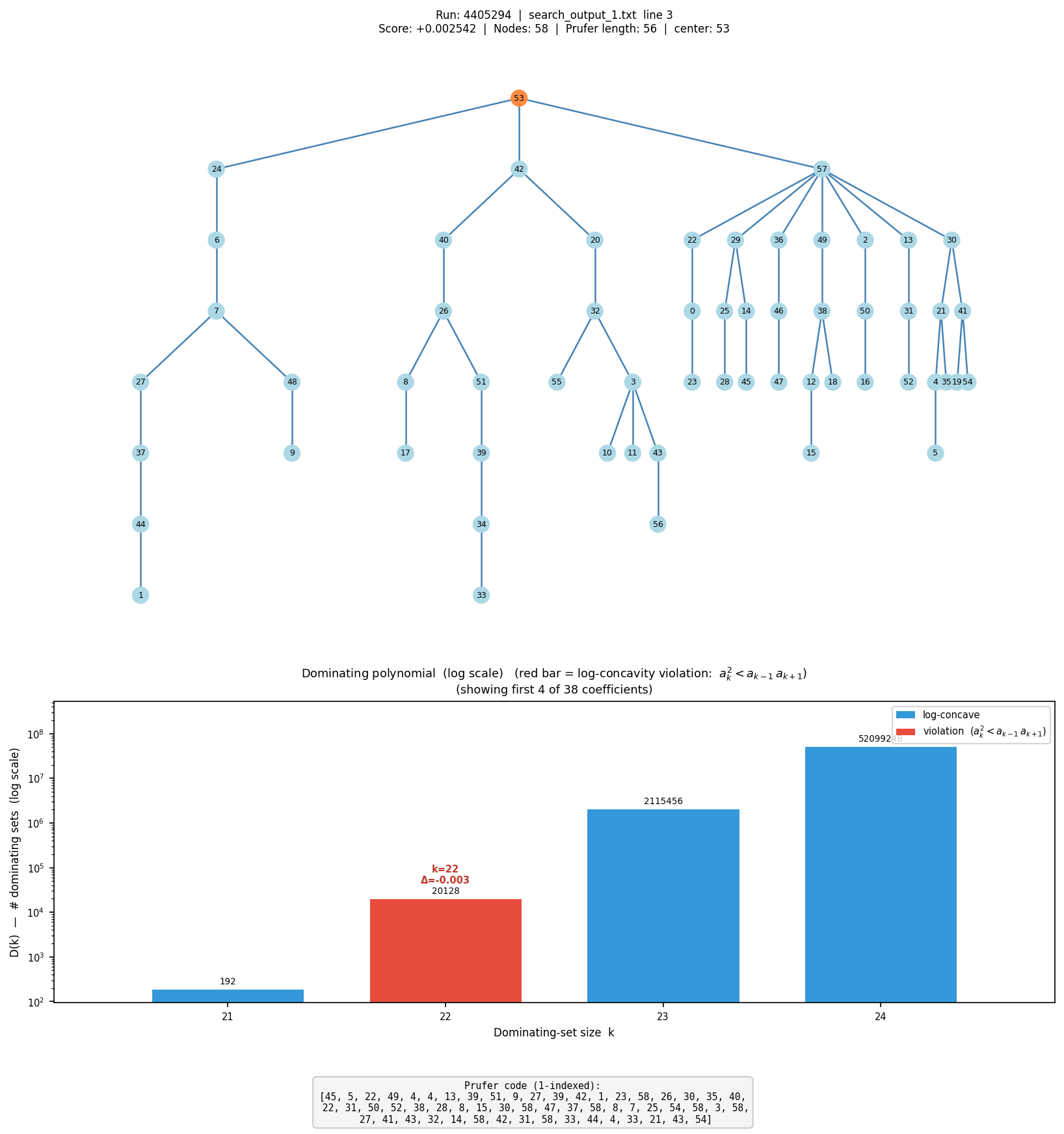}
  \caption{A 58-vertex tree with a non log-concave dominating set sequence}
    \label{fig5}
\end{figure}

\begin{figure}[ht!]
    \centering
    \includegraphics[width=\textwidth]{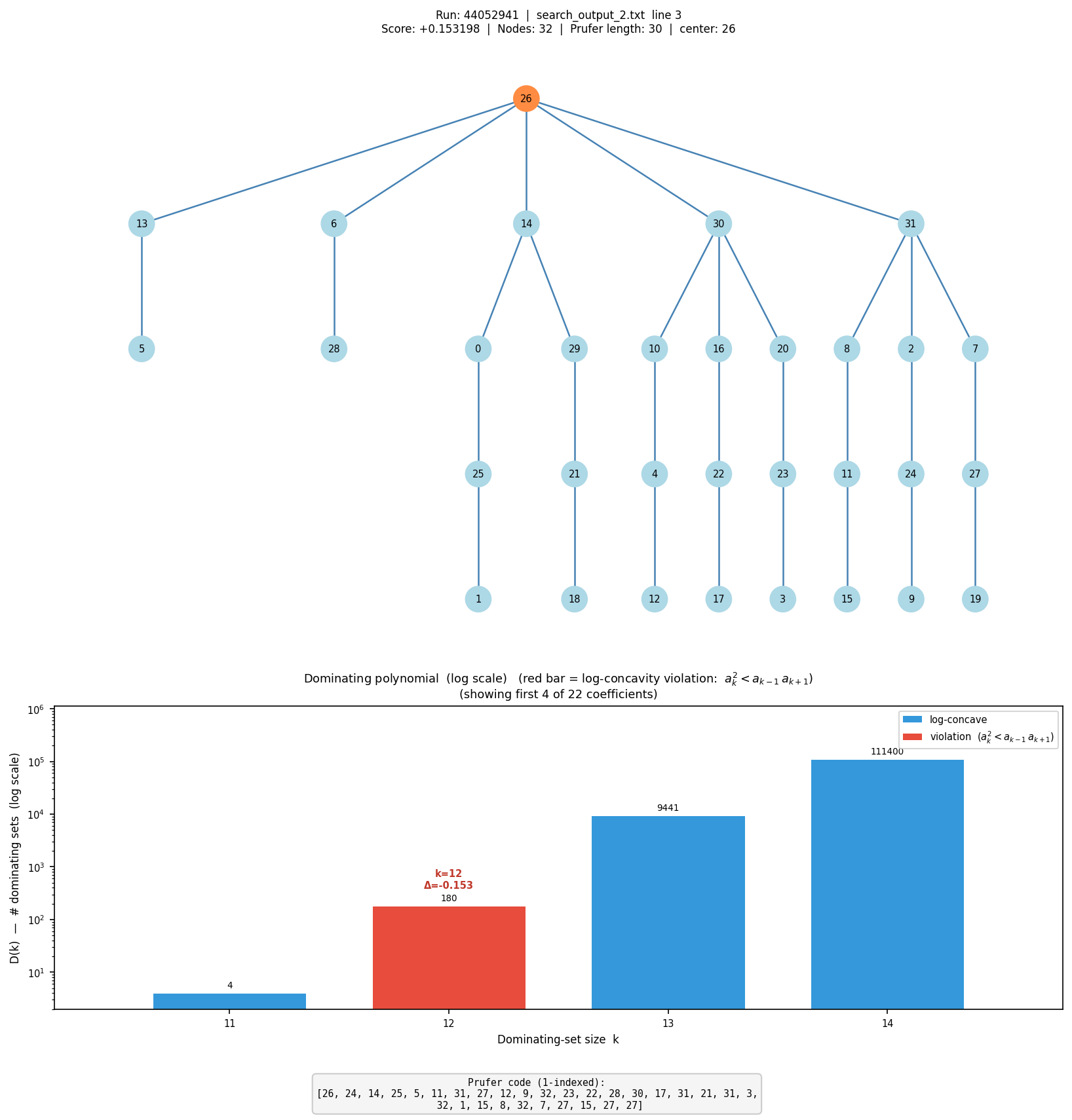}
  \caption{A 32-vertex tree with a non log-concave dominating set sequence}
    \label{fig6}
\end{figure}

\clearpage






\clearpage

\textbf{Acknowledgement}.  Thanks to Eric Ramos for detailed discussions on the PatternBoost codes.  Thanks also to Ferenc Bencs, David Galvin and Elchanan Mossel for helpful discussions.  Thanks also to Geordie Williamson for suggesting to use Patternboost on our problem.

\providecommand{\bysame}{\leavevmode\hbox to3em{\hrulefill}\thinspace}
\providecommand{\MR}{\relax\ifhmode\unskip\space\fi MR }
\providecommand{\MRhref}[2]{%
  \href{http://www.ams.org/mathscinet-getitem?mr=#1}{#2}
}
\providecommand{\href}[2]{#2}


\begin{thebibliography}{HHK{\etalchar{+}}25}


\bibitem[AHK18]{karim18}
Karim Adiprasito, June Huh and Eric Katz.
\emph{Hodge theory for combinatorial geometries}.  Annals of Mathematics. \textbf{2} (2018), vol. 188, pp. 381--452.

\bibitem[AMSE87]{alavi87}
Yousef Alavi, Paresh J. Malde, Allen J. Schwenk, and Paul Erd\H{o}s. \emph{The vertex independence sequence of a graph is not constrained}, Eighteenth Southeastern International Conference on Combinatorics, Graph Theory, and Computing (Boca Raton, Fla.), (1987) vol. 58, pp. 15--23.

\bibitem[AP14]{ali14}
S. Alikhani and Y. H. Peng. \emph{Introduction to domination polynomial of a graph}. Ars Combinatoria, \textbf{5} (2014), vol. 114, pp. 257–-266.


\bibitem[BES18]{BES2018}
P.~Bahls, B.~Ethridge and L.~Szabo.
\emph{Unimodality of the independence polynomials of non-regular caterpillars},
Australas. J. Combin. \textbf{1} (2018), vol. 71, pp. 104--112.


\bibitem[B25]{bautista25}
C\'{e}sar Bautista-Ramos.
\emph{Multiple breaks of log-concavity in the independence polynomials of trees}.
Preprint (2025), \href{https://arxiv.org/abs/2511.00334}{arXiv:2511.00334}.

\bibitem[B21]{beaton21}
I. Beaton. On dominating sets and the domination polynomial. Ph.D. Thesis. 2021.

\bibitem[BB22]{beaton22}
Beaton, I., Brown, J.I. \emph{On the Unimodality of Domination Polynomials}. Graphs and Combinatorics \textbf{3} (2022), vol. 38, pp. 1--11.

\bibitem[BG21]{basit21}
Abdul Basit and David Galvin.  \emph{On the Independent Set Sequence of a Tree}. The Electronic Journal of Combinatorics. \textbf{3} (2021), vol. 28. pp. 1--23.

\bibitem[BS24]{beaton24}
Iain Beaton and Sam Schoonhoven.
\emph{On the unimodality of nearly well-dominated trees}.
Preprint (2024), \href{https://arxiv.org/abs/2411.02288}{arxiv:2411.02288}.

\bibitem[CEWW24]{charton24}
Fran\c{c}ois Charton, Jordan S. Ellenberg, Adam Zsolt Wagner, Geordie Williamson.
\emph{PatternBoost: Constructions in Mathematics with a Little Help from AI}.
Preprint (2024), \href{https://arxiv.org/abs/2411.00566}{arxiv:2411.00566}.


\bibitem[GH18]{galvin2018}
David~Galvin and Justin~Hilyard.
\emph{The independent set sequence of some families of trees}, Australas. J. Combin., \textbf{2} (2018), vol. 70, pp. 236--252.

\bibitem[HKM10]{harsha10}
Prahladh Harsha, Adam Klivans and Raghu Meka. \emph{An invariance principle for polytopes}.  Proceedings of the Forty-Second ACM Symposium on Theory of Computing (2010).  pp. 543–552.

\bibitem[Hei25]{heilman25}
Steven Heilman. \emph{Independent Sets of Random Trees and Sparse Random Graphs}.  Journal of Graph Theory, \textbf{3} (2025), vol. 109, pp. 1--28.

\bibitem[KL23]{kadrawi23}
Ohr Kadrawi and Vadim E. Levit. 
\emph{The independence polynomial of trees is not always log-concave starting from order 26}. Ars Mathematica Contemporanea. \textbf{4} (2025), vol. 25, pp. 1--23.

\bibitem[KPT17]{kotek17}
Tomer Kotek, James Preen and Peter Tittmann.  \emph{Domination Polynomials of Graph Products}.  Journal of Combinatorial Mathematics and Combinatorial Computing. \textbf{1} (2017), vol. 101, pp. 245--258. 

\bibitem[RS25]{ramos25}
Eric Ramos, Sunny Sun.
\emph{An AI enhanced approach to the tree unimodality conjecture}.
Preprint (2025), \href{https://arxiv.org/abs/2510.18826}{arXiv:2510.18826}.

\bibitem[WZ11]{WangZhu2011}
Yi~Wang and Bao-Xuan~Zhu.
\emph{On the unimodality of independence polynomials of some graphs}.
European J. Combin. \textbf{1} (2011), vol. 32, pp. 10--20.

\bibitem[Zhu07]{Zhu2007}
Zhi-Feng~Zhu.
\emph{The unimodality of independence polynomials of some graphs}.
Australas. J. Combin. \textbf{1} (2007), vol. 38, pp. 27--33.



\end{thebibliography}
\end{document}